\documentclass[11pt]{amsart}
\usepackage{amsmath}
\usepackage{amsthm}
\usepackage{amssymb}
\newtheorem{theorem}{Theorem}
\newtheorem{lemma}[theorem]{Lemma}
\newtheorem{corollary}[theorem]{Corollary}
\theoremstyle{remark}

\DeclareMathOperator{\Cay}{Cay}

\begin{document}
\title{Diameter 2 Cayley Graphs of Dihedral Groups}
\author{Grahame Erskine}
\address{Department of Mathematics and Statistics, The Open University, Milton Keynes, UK}
\email{Grahame.Erskine@open.ac.uk}
\keywords{Degree-diameter problem, Cayley graph, Dihedral group}
\subjclass[2010]{05C25}
\begin{abstract}
We consider the degree-diameter problem for Cayley graphs of dihedral groups.
We find upper and lower bounds on the maximum number of vertices of such a graph with diameter 2 and degree $d$.
We completely determine the asymptotic behaviour of this class of graphs by showing that both limits are asymptotically $d^2/2$.
\end{abstract}
\maketitle

\section{Introduction}
The \emph{degree-diameter problem} seeks to determine the largest possible number $n(d,k)$ of vertices of a graph of maximum
degree $d$ and diameter $k$. The survey \cite{miller2005moore} summarises current known results for the general case and also for various restricted problems
where we consider only graphs of certain classes. In this paper we restrict attention to Cayley graphs of dihedral groups with diameter 2 and prove an asymptotic
limit for the maximum order of such graphs.

An elementary counting argument yields an upper limit of $d^2+1$ (the well-known \emph{Moore bound}) for a graph of maximum degree $d$ and diameter 2. 
Thus for any family of diameter 2 graphs the largest possible asymptotic order is $d^2$. When we restrict consideration to Cayley graphs,
it is known \cite{siagiova2012approaching} that there is a relatively sparse family of groups (affine groups over finite fields of characteristic 2)
on which Cayley graphs can be constructed with asymptotic order $d^2$. 
For more elementary families of groups, less is known. For example an extension of the counting argument shows that for abelian
groups the largest possible asymptotic order is $d^2/2$, but the best known result \cite{macbeth2012cayley} has limit $3d^2/8$.
A recent result of Abas\cite{Abas2014}, shows that a Cayley graph of diameter 2 and asymptotic order $d^2/2$ can be constructed for any degree $d$
using direct products of dihedral and cyclic groups.

In this paper we show that the asymptotic limit for dihedral groups is precisely $d^2/2$, 
first by obtaining a lower bound by way of a construction involving Galois fields, 
and then by finding an upper bound for generalised dihedral groups by a counting argument.

We denote the dihedral group of order $n$ by $D_n$. We will view the usual dihedral group as an example of a 
\emph{generalised dihedral group} $G\rtimes C_2$ which is a semidirect product of an abelian group $G$
with the multiplicative group $\{\pm 1\}$ where the action on $G$ is via its inversion automorphism.
For a group $G$ and a subset $S\subset G$ which is inverse-closed and does not contain the identity in $G$, we define the
Cayley graph $\Cay(G,S)$ to have vertex set $G$ and an undirected edge $gh$ if and only if $gh^{-1}\in S$. The graph is vertex-transitive and hence regular,
with degree $d=\lvert S\rvert$. From the definition of Cayley graphs it is immediate that $\Cay(G,S)$ has diameter at most $k$ if and only if each element of $G$
can be expressed as a product of no more than $k$ elements of the generating set $S$.

By $DC(d,k)$ we mean the largest number of vertices in a Cayley graph of a dihedral group having degree $d$ and diameter $k$.

\section{Results}
Our first result uses a construction based on finite fields to obtain a lower bound for $DC(d,2)$ for certain values of $d$.
The method is similar to constructions in \cite{macbeth2012cayley}. We also use the well-known result that
the cyclic group $\mathbb{Z}_n$ has a diameter 2 Cayley graph with a generating set of size at most $2\lceil\sqrt{n}\rceil$.
(To see this, let $K=\lceil\sqrt{n}\rceil,M=\lfloor\frac{K}{2}\rfloor$ and take a generating set consisting of 
$\{\pm 1,\pm 2,\ldots,\pm M,\pm K,\pm 2K,\ldots,\pm MK\}$.)
\begin{lemma}
\label{lemma:lbound}
If $p$ is any prime and $d=2(p+\lceil\sqrt{p}\rceil-1)$, then $DC(d,2)\geq 2p(p-1)$.
\end{lemma}
\begin{proof}
Let $F$ be the Galois field $GF(p)$ where $p$ is a prime. The additive and multiplicative groups $F^+$ and $F^*$ are cyclic of coprime orders
so that $F^+\times F^*$ is a cyclic group of order $n=p(p-1)$.
Consider the dihedral group $D_{2n}$ as a semidirect product 
$G=(F^+\times F^*)\rtimes C_2$ where the cyclic group $C_2$ is thought 
of as the multiplicative group $\{\pm 1\}$ and acts on $F^+\times F^*$ by inversion. 
Specifically, the multiplication rule is:
\[(a,b,c)(\alpha,\beta,\gamma)=(a+\alpha c,b\beta^c,c\gamma)\]

The subgroup $C=\langle(1,1,1)\rangle$ is cyclic of order $p$ and so itself has a diameter 2 Cayley graph with respect to some generating set $\{c_i\}$
of cardinality $2\lceil\sqrt{p}\rceil$.
Consider now a generating set $S$ of the full group $G$ containing:
\begin{align*}
v=(0,1,-1) & \qquad(1\text{ element})\\
a_x=(0,x,1),x\in F^*\setminus\{1\} & \qquad(p-2\text{ elements})\\ 
b_x=(x,x,-1),x\in F^* & \qquad(p-1\text{ elements})\\
c_i,i=1\ldots 2\lceil\sqrt{p}\rceil & \qquad(2\lceil\sqrt{p}\rceil\text{ elements})
\end{align*}
It suffices to show that the set $S$ is inverse-closed and that every element of the group can be expressed as the product of at most two of these generators.

Since $v^{-1}=v,a_x^{-1}=a_{x^{-1}},b_x^{-1}=b_x$ and $\{c_i\}$ is inverse-closed it follows that $S$ is inverse-closed.
To show that the diameter is 2 we consider all the possible cases as follows.

If $x\neq 0,x\neq y$ then $(x,y,-1)=(0,z,1)(x,x,-1)=a_z b_x$ where $z=yx^{-1}$. 

If $x\neq 0,x=y$ then $(x,y,-1)=(x,x,-1)=b_x$. 

If $x=0,y\neq 1$ then $(x,y,-1)=(0,y,1)(0,1,-1)=a_y v$.

If $x=0,y=1$ then $(x,y,-1)=(0,1,-1)=v$.

If $y\neq 1,x\neq 0$ then $(x,y,1)=(z,z,-1)(t,t,-1)=b_z b_t$ where $z=yx(y-1)^{-1},t=x(y-1)^{-1}$. 

If $y\neq 1,x=0$ then $(x,y,1)=(0,y,1)=a_y$. 

If $y=1$ then $(x,y,1)\in C$ and so is the product of at most two $c_i$.

Since $\lvert S\rvert=2(p+\lceil\sqrt{p}\rceil-1)$ the result follows.
\end{proof}

The previous result shows that $\displaystyle\limsup_{d\to\infty}\frac{DC(d,2)}{d^2}\geq\frac{1}{2}$.
The next result shows that $1/2$ is in fact also an upper bound.
\begin{lemma}
\label{lemma:ubound}
Let $G$ be a generalised dihedral group of order $2n$ and let $S$ be an inverse-closed generating set for $G$ not containing the identity.
Suppose that the Cayley graph $\Cay(G,S)$ has diameter 2. Then the degree $d$ of $\Cay(G,S)$ satisfies $d\geq 2\sqrt{n}-1$.
\end{lemma}

\begin{proof}
Let $G=H\rtimes C_2$ where $H$ is an abelian group of order $n$ and $C_2$ acts on $H$ by inversion. 
Let $C$ be the index 2 subgroup of $G$ isomorphic to $H$ and write $S=A\cup B$ where $A\subset C$ and $B\subset G\setminus C$.
Let $m_1=\lvert A\rvert,m_2=\lvert B\rvert$.

Consider how the $n$ elements of $G\setminus C$ can be expressed as a product of at most two elements in $S$. There are $m_2$
possibilities from the set $B$ itself, then $m_1 m_2$ elements of the form $ab$ where $a\in A,b\in B$. Since $a^{-1}b=ba$ and the set
$A$ is inverse-closed the products of the form $ba$ do not contribute any further elements. So we require:
\[m_2(m_1+1)\geq n\]
The degree $d$ of the Cayley graph is $\lvert S\rvert=m_1+m_2$. All numbers are inherently positive and so elementary calculus shows that the minimum
possible value of $m_1+m_2$ occurs when $m_2=m_1+1=\sqrt{n}$. So $d\geq 2\sqrt{n}-1$.
\end{proof}

The bound $|G|\leq\frac12(d+1)^2$ in Lemma \ref{lemma:ubound} is valid for all values of $d$, 
but as it stands Lemma \ref{lemma:lbound} only holds for a restricted set of values.
We can extend the result of Lemma \ref{lemma:lbound} by using the ideas in \cite{siran2011large} to obtain a lower bound valid for all values of $d$.
\begin{lemma}
\label{lemma:extend}
Let $d\geq 6$ and let $p$ be the largest prime satisfying $D(p)=2(p+\lceil\sqrt{p}\rceil-1)\leq d$.
Then $DC(d,2)\geq 2p(p-1)$.
\end{lemma}

\begin{proof}
Let $p$ be as in the statement, $n=2p(p-1)$ and $G=D_{n}$. 
By Lemma \ref{lemma:lbound} there is an inverse-closed unit-free subset $S\subset G$ with $\lvert S\rvert=D(p)$ such that
$\Cay(G,S)$ has diameter 2. We can add $d-D(p)$ involutions from $G\setminus S$ to form a new inverse-closed unit-free generating set $S'$. 
The diameter of $\Cay(G,S')$ is still 2 and the result follows.
\end{proof}

Using the method of \cite{siran2011large} we may use ideas from number theory to obtain a result independent of $p$. Specifically, from the result in 
\cite{baker2001difference} it follows that for all sufficiently large $D$, there is some prime $p$ in the range $D-D^{0.525}\leq p\leq D$.
\begin{theorem}
\label{thm:main}
For all sufficiently large $d$, $DC(d,2)\geq 0.5 d^2-1.39d^{1.525}$.
\end{theorem}

\begin{proof}
For given $d$, let $p$ be the largest prime such that $2(p+\lceil\sqrt{p}\rceil-1)\leq d$. 
Then $p$ is at least as large as the largest prime $q$ satisfying $2(q+\sqrt{q})\leq d$. 
Rearranging this we find that $q$ is the largest prime not exceeding $D=\frac12(d-\sqrt{2d+1}+1)$. 
By the result of \cite{baker2001difference} for sufficiently large $d$ we have $q\geq D-D^{0.525}$.
So for large $d$ we have:
\begin{align*}
p\geq q&\geq D-\left(\frac{d}{2}\right)^{0.525}\\
&=\frac12\left(d-(2d+1)^{0.5}+1-2^{0.475}d^{0.525}\right)
\end{align*}

For large $d$ the term in $d^{0.525}$ dominates terms of lower powers of $d$ and since $2^{0.475}\approx 1.389918$, for large $d$ we have
$p\geq\frac12\left(d-1.38992d^{0.525}\right)$. For sufficiently large $d$ we therefore have
\[2p(p-1)\geq \frac{d^2}{2}-1.39d^{1.525}\]
Since we can construct a Cayley graph of degree $d$ and diameter 2 on the dihedral group of order $2p(p-1)$ by Lemma \ref{lemma:extend} the result follows.
\end{proof}

Lemma \ref{lemma:ubound} and Theorem \ref{thm:main} allow us to determine completely the asymptotic behaviour of $DC(d,2)$.
\begin{corollary}
\[\lim_{d\to\infty}\frac{DC(d,2)}{d^2}=\frac{1}{2}\]
\end{corollary}
\section{Remarks}
It is tempting to try to extend these results to other split extensions of abelian groups where the action is via an automorphism other than the inversion map.
However the counting argument of Lemma \ref{lemma:ubound} relies on the fact that in our Cayley graph $\Cay(G,S)$ the generating set $S$ 
has the very particular form $\{(a,1):a\in S_1\}\cup\{(b,-1):b\in S_2\}$ where the set $S_1$ is inverse-closed and hence closed under the acting automorphism. 
The upper asymptotic bound $d^2/2$ would not necessarily hold for other more general semidirect products of an abelian group with
$C_2$, although no family of such groups with a larger bound is known.

We illustrate this remark with a couple of examples. Firstly, the construction of Abas\cite{Abas2014}, uses a direct product of the form $D_{2m}\times\mathbb{Z}_n$
which we may regard as the semidirect product $(\mathbb{Z}_m\times\mathbb{Z}_n)\rtimes C_2$, where $C_2$ acts on $\mathbb{Z}_m$ via its inversion automorphism
and on $\mathbb{Z}_n$ via the identity automorphism. In this case a generating set $S$ which is inverse-closed is not necessarily of the form 
in the previous paragraph, since the action in the semidirect product is not inversion.

However, we may modify the construction to obtain a family of groups for which the counting argument of Lemma \ref{lemma:ubound} does hold.
In the first example above, if $n=2$ then the identity automorphism coincides with the inversion automorphism and the Lemma applies.
Further, if we replace $\mathbb{Z}_2$ by any elementary abelian 2-group then its inversion automorphism is the identity and the argument continues to hold.
Thus the upper bound of Lemma \ref{lemma:ubound} holds for any group of the form $D_{2n}\times H$ where $H$ is an elementary abelian 2-group.
\bibliographystyle{amsplain}
\bibliography{grahame}
\end{document}